\documentclass{amsart}

\theoremstyle{plain}

\newtheorem{thm}{Theorem}[section]
\newtheorem{thmm}{Theorem}
\newtheorem{cor}[thm]{Corollary}
\newtheorem{conj}{Conjecture}
\newtheorem{lem}[thm]{Lemma}
\newtheorem{prop}[thm]{Proposition}

\newcommand{\Aut}{\mbox{\rm Aut}}
\newcommand{\Adj}{\mbox{\rm Adj}}
\newcommand{\Tr}{\mbox{\rm Tr}}

\begin{document}

\title{Large connected strongly regular graphs are
  Hamiltonian}
\author{L. Pyber}

\maketitle

\begin{abstract}
  We prove that every connected strongly regular graph
  on sufficiently many vertices contains a Hamiltonian
  cycle. We prove this by showing that, apart from
  three families, connected strongly regular graphs are
  (highly) pseudo-random. Our results suggest a number of new questions and conjectures.

  We also show that dense graphs with primitive automorphism groups are
  pseudo-random unless their automorphism groups are
  almost abelian.
\end{abstract}

%%OK

\setcounter{section}{-1}
\section{Introduction}\label{sec:null}

A graph is \textit{strongly regular} with parameters
$(v,k,\lambda,\mu)$ if it is a $k$-regular graph on $v$
vertices such that

\begin{enumerate}
\item[(i)] each edge is in $\lambda$ triangles,

\item[(ii)] any two non-adjacent vertices have $\mu$
  common neighbours.
\end{enumerate}
(We exclude the complete and  empty graphs).

For example, the Petersen graph is strongly regular with
parameters (10,3,0,1).

It is conjectured by Brouwer and Haemers~\cite{BH2}
that the Petersen graph is the only connected
non-Hamiltonian strongly regular graph.

They proved the conjecture for graphs on at most 98
vertices and remarked that it can be shown to hold for
graphs with "most admissible parameter sets for
strongly regular graphs" by means of a sufficient
condition for hamiltonicity given in~\cite{BH2}.  It is
certainly not hard to show that the members of several
infinite families of strongly regular graphs are
Hamiltonian. These include complete multipartite graphs
of type $K_{m,m,\dots,m}$, line graphs of Steiner
systems and Latin square graphs (these families will be
discussed later).

We prove the following.

\begin{thmm}\label{th:01}
  There is a number $N$ such that if $G$ is a connected
  strongly regular graph on $v > N$ vertices, then $G$
  contains a Hamiltonian cycle.
\end{thmm}

Connected strongly regular graphs are exactly distance regular graphs of diameter 2 . Motivated by Theorem~\ref{th:01} and some additional evidence (discussed in Section 3) we make  the following (somewhat provocative) conjecture.

\begin{conj}\label{conj:02}
Connected distance regular graphs are Hamiltonian with finitely many exceptions.
\end{conj}

If $G$ is a $k$-regular graph then the eigenvalues of
its adjacency matrix are real numbers $k=\lambda_1
\geq\lambda_2 \geq\dots \geq\lambda_v$. Set
$\Lambda=\Lambda(G)=\max \{|\lambda_i(G)| \mid i={2,3,\dots,
  v}\}$. The parameter $\Lambda$ is usually called the
second eigenvalue of $G$.

Our proof rests on the following hamiltonicity
criterion of Krivelevich and Sudakov~\cite{KS2}.

\begin{thmm}[Krivelevich, Sudakov]\label{th:03}
  Let $G$ be a $k$-regular graph on $v$ vertices. If
  $v$ is large enough and $k/\Lambda> 1000\log
  v(\log\log\log v)/(\log\log v)^2$, then $G$ is
  Hamiltonian.
\end{thmm}

Let us call a connected strongly regular graph
\textit{exceptional} if it is not a complete
multipartite graph or a Steiner graph or a Latin square
graph. It seems to be impossible to classify exceptional strongly regular graphs. The central result of this note is the
following.

\begin{thmm}\label{th:04}
  Let $G$ be an exceptional strongly regular graph.
  Then $k/\Lambda > v^{1/10}/2$.
\end{thmm}

By the above discussion Theorem~\ref{th:01} follows.
One can give an astronomical but explicit estimate for
the value of $N$ in Theorem~\ref{th:01} by glancing at
the proof of Theorem~\ref{th:03}.  To prove that
connected strongly regular graphs on say, at least
$1000$ vertices are Hamiltonian seems to be a
challenging task.
%%% ide jon valamî?

Theorem~\ref{th:04} shows that exceptional strongly
regular graphs are highly pseudo-random. Pseudo-random graphs,
that is graphs that have edge distribution similar to
that of a truly random graph were first systematically
investigated by Thomason~\cite{Th1}, \cite{Th2}.
Thomason already noted that many families of strongly
regular graphs have certain pseudo-random properties.This observation is made more explicit in the recent survey ~\cite{KS1}.

As explained e.g. in~\cite{KS1}   if
$G$ is a $k$-regular graph and $k/\Lambda$ is large then
$G$ has almost uniform edge distribution i.e. it is
pseudo-random (we do not give a formal definition, for
various possibilities see~\cite{AS}, ~\cite{KS1}.) In~\cite{BL} it
is shown that on the other hand if a $k$-regular graph
has almost uniform edge-distribution then $k/\Lambda$
is large.

We remark that the eigenvalues of the non-exceptional strongly regular
graphs are well-known. Hence Theorem~\ref{th:04}
largely clarifies the connection between
pseudo-randomness and strongly regular graphs discussed
e.g.\ in~\cite{Th2}, \cite{KS1} and~\cite{Ni}.

Using related ideas we exhibit other large, natural
classes of pseudo-random graphs. For example we prove
the following.

\begin{thmm}\label{th:05}
  Let $G_n$ be sequence of $k_n$-regular graphs of
  order $v_n$ with $k_n > \epsilon v_n$ (for some
  $\epsilon > 0$). Assume that the groups
  $X_n=\Aut(G_n)$ are primitive permutation groups and
  that the index of the largest abelian normal subgroup
  of $X_n$ goes to infinity.  Then
  $\Lambda(G_n)=o(k_n$), hence $G_n$ is pseudo-random.
\end{thmm}

{Acknowledgement

I would like to thank Bal\'{a}zs Szegedy for several useful conversations in the ,,M\'{u}zeum''.}

\setcounter{section}{0}
\section{General bounds}\label{sec:one}

We call a strongly regular graph~$G$
\textit{primitive\/} if both $G$ and its complement are
connected (i.e.\ if~$G$ is not a complete multipartite
graph of type $K_{m,m,\dots,m}$ or the complement of
such a graph).  In this section we will derive bounds
for the eigenvalue ratios of primitive strongly regular
graphs using well-known results.  Apart from the
eigenvalue $k$ a connected strongly regular graph $G$
has two distinct eigenvalues which satisfy $r >0$ and
$-1 > s$ (hence $\Lambda =\max\{r,-s\}$). Denote by $f$
and $g$ their respective multiplicities.

We need the following basic restrictions on the
parameters.

\begin{prop}\label{pr:11}
Let $G$ be a strongly regular graph. Then we have

\begin{itemize}
\item[(i)]   $k^2+fr^2+gs^2=kv$,
\item[(ii)]  $k(k-1-\lambda)=\mu(v-k-1)$,
\item[(iii)] $rs=\mu-k$,
\item[(iv)]  $r+s=\lambda-\mu$,
\item[(v)]   $r$ and $s$ are integers, except perhaps
  when $G$ is a conference graph, that is,
  $(v,k,\lambda,\mu)=(4t+1,2t,t-1,t)$ for some integer
  $t$, $r=(\sqrt{4t+1}{-1})/2$ and
  $s=(-\sqrt{4t+1}-1)/2$.
\end{itemize}
\end{prop}

The first equation follows from $k^2+fr^2+gs^2 =
\Tr(A^2)=kv$, where $A$ denotes the adjacency matrix of
$G$. For the rest see e.g.~\cite{BH1}.

Seidel's absolute bound~\cite{Se} will provide us with
a crucial starting point.

\begin{thm}[Seidel]\label{th:12}
  Let $G$ be a primitive strongly regular graph. Then
  we have

\begin{itemize}
 \item[(i)]  $v\leq f(f+3)/2$,
 \item[(ii)] $v\leq g(g+3)/2$.
 \end{itemize}
\end{thm}

\begin{lem}\label{lem:13}
 Let G be a primitive strongly regular graph. Then we have \\
 $k/\Lambda >\sqrt{k/\sqrt v}$
\end{lem}

\begin{proof}
  Our statement is equivalent to $\Lambda^4 < vk^2$.
  We prove $r^4 < vk^2$ (a similar argument shows that
  $s^4 < vk^2$). Theorem~\ref{th:12} implies that $v-1
  \leq f^2$. Proposition~\ref{pr:11}~(i) implies that
  $fr^2 \leq k(v-k)$. Hence $r^4 < vk^2 (v-k)/f^2 \leq
  vk^2$ as required.
\end{proof}

Lemma~\ref{lem:13} combined with Theorem~\ref{th:02}
already shows that large enough connected strongly
regular graphs with $k > \sqrt v(1000\log v)^2$ are
Hamiltonian. Note that we have $k>\sqrt{v-1}$ for any
connected strongly regular graph.

Another consequence of Lemma~\ref{lem:13} is the
following.

\begin{cor}\label{co:14}
Let $G$ be a primitive strongly regular graph. Then we have \\
$|\lambda-\mu| < v^{3/4}$.
\end{cor}

\begin{proof}
  Using Proposition~\ref{pr:11}~(iv) we obtain that
  $|\lambda-\mu|\leq \Lambda < v^{1/4} \sqrt{k}\leq
  v^{3/4}$ as required.
\end{proof}

Earlier it was shown by Nikiforov~\cite{Ni} using
Szemer\'{e}di's Regularity lemma that
$|\lambda-\mu|=o(v)$, hence strongly regular graphs of
degree $> \varepsilon v$ are pseudo-random for any
fixed $\varepsilon > 0$. This was a useful hint.

Next we establish half of Theorem~\ref{th:03}.

\begin{lem}\label{lem:15}
  Let $G$ be a primitive strongly regular graph. Then
  we have
 $$
\frac{k}{|s|} > \frac{v^{1/6}}{2}.
$$
\end{lem}

\begin{proof}
  Denote $|s|=-s$ by $m$.  If $k > v/4$ then $k/m >
  v^{1/4}/2$ follows from Lemma~\ref{lem:13}.

  By Proposition~\ref{pr:11}~(iv) we have
  $m=-s=r+\mu-\lambda \leq r+\mu$.

  If $r \geq\mu$ then this implies $m \leq 2r$. By
  Proposition~\ref{pr:11}~(iii) we have $rm=-rs=$ $=k-\mu
  \leq k$. It follows that $m \leq \sqrt{2k}$. Using
  $k^2+1 \geq v$ we obtain that $k/m > \sqrt{k/2} \geq
  (v-1)^{1/4}/\sqrt{2} \geq v^{1/6}/2$.

  Assume finally that $4k \leq v$ and $\mu > r$.  Then
  using Proposition~\ref{pr:11}~(ii) we see that $m
  \leq r+\mu < 2\mu= 2k(k-1-\lambda)/(v-k-1) <
  2k^2/(v-k)$ and hence $k/m > (v-k)/2k$.  On the other
  hand Lemma~\ref{lem:13} implies that $m^4 < vk^2$ and
  hence $(k/m)^4 > k^2/v$.  Multiplying we obtain that
  $(k/m)^6 > (v-k)^2/4v$.  Using $v \geq 4k$ we see
  that $(k/m)^6 > v/2^6$ as required.
\end{proof}

\section{Exceptional vs non-exceptional
  graphs}\label{sec:two}

We now discuss the three classes of connected
non-exceptional graphs (see~\cite{Go2} for more
details)

\begin{itemize}
\item[(i)] Complete multipartite graphs of type
  $K_{m,m,\dots,m}$ with $n$ parts of size $m$.  The
  eigenvalues of these graphs are $m(n-1),0,-m$,

\item[(ii)] Latin square graphs: given $m-2$ mutually
  orthogonal Latin squares of order $n$, the vertices
  are the $n^2$ cells, two vertices are adjacent if
  they lie in the same row or column or have the same
  entry in one of the squares. We allow $m=2$ (in which
  case the graph is isomorphic to the line graph of
  $K_{m,m}$).  The eigenvalues of these graphs are
  $m(n-1),n-m,-m$,

\item[(iii)] Steiner graphs:the vertices are the blocks
  of a Steiner system $S(2,m,n)$ (that is, they are
  $m$-subsets of an $n$-set with the property that any
  two elements of the set lie in a unique block): two
  vertices are adjacent if and only if the
  corresponding blocks intersect. We allow  $m=2$. The
  eigenvalues of these graphs are $m((n-1)/(m-1)-1),
  (n-1/m-1)-m-1,-m$.
\end{itemize}

It is an easy exercise to show that graphs in the first
two classes are Hamiltonian. For the third class
see~\cite{HR}. It is clear from the above that for
certain types of strongly regular graphs such as line
graphs of Steiner triple systems $k/\Lambda$ can be
quite small and hence these graphs are far from
pseudo-random.

To complete the proof of Theorem~\ref{th:03} we will
need a well-known deep result of Neumaier~\cite{Ne}.

\begin{thm}[Neumaier]\label{th:21}
  Let $G$ be an exceptional strongly regular graph with
  s an integer. Then we have $r \leq
  s(s+1)(\mu+1)/2-1$.
\end{thm}

\begin{lem}\label{lem:22}
  Let $G$ be an exceptional strongly regular graph.
  Then we have\\ $k/r > v^{1/10}$.
\end{lem}

\begin{proof}
  If $G$ is a conference graph then our statement
  follows by direct computation. Hence we can assume
  that s is an integer and $s \leq -2$. Set $m=-s$.

  If $k > v^{7/10}$ then our statement follows from
  Lemma~\ref{lem:13}.

  By Proposition~\ref{pr:11}~(iii) we have $k=rm+\mu$.
  Hence $k/r \geq m$ and if $m > v^{1/10}$ our
  statement follows.

  Furthermore if $m \geq\mu$ then using
  Theorem~\ref{th:21} we obtain that $\sqrt{v-1}\leq
  k=rm+\mu \leq \frac{1}{2}(m^2(m-1)(\mu+1))-m+\mu < m^4/2$ and
  therefore $m > v^{1/8}$.

  Assume now that $m < \mu, k \leq v^{7/10}$ and $m <
  v^{1/10}$.  This implies in particular that $v >
  2^{10}$ and $v > 8k$.

  As above we have $k=rm+\mu \leq
  1/2(m^2(m-1)(\mu+1))-m+\mu$ hence $k \leq (m^3)\mu/2
  +\mu =(m^3+2)\mu/2$.

  By Proposition~\ref{pr:11}~(ii) we have $\mu \leq
  k(k-1)/(v-k-1) < k^2/(v-k)$.  This implies $k \leq
  (m^3+2)k^2/2(v-k$) and hence $2(v-k)\leq (m^3+2)k$.
  Using our assumptions we obtain that $2(7v/8) \leq
  ((m^3+2)/m^3)(km^3) < (10/8)v$, a contradiction.
\end{proof}

The above lemma was inspired by a result of
Spielman~\cite[Corollary 9]{Sp}, which says that if $G$
is an exceptional strongly regular graph with $k=o(v)$
then $r=o(k)$. Spielman used his result as a starting
point for an algorithm testing isomorphism of strongly
regular graphs in time $v^{0(v^{1/3\log v})}$

\smallskip

Lemma~\ref{lem:22} completes the proof of
Theorem~\ref{th:03}.

As noted above connected non-exceptional strongly
regular graphs are Hamiltonian. Combining
Theorem~\ref{th:02} and Theorem~\ref{th:03} we obtain
that large enough exceptional strongly regular graphs
are also Hamiltonian. This completes the proof of
Theorem~\ref{th:01}.

In fact one can see that if $G$ is an exceptional
strongly regular graph of degree $k$ with $v$
sufficiently large then deleting roughly $k/2$ edges
from each vertex in an arbitrary way we still obtain a
Hamiltonian graph. This follows immediately from
Theorem~\ref{th:03} and a qualitative strengthening of
Theorem~\ref{th:02} due to Sudakov and Vu~\cite{SV}.

\begin{thm}[Sudakov, Vu]\label{th:23}
  For any fixed $\varepsilon > 0$ and $v$ sufficiently
  large the following holds. Let $G$ be a $k$-regular
  graph on $v$ vertices such that $k/\Lambda > (\log
  n)^2$. If $H$ is a subgraph of $G$ with maximum
  degree $\Delta(H)\leq (1/2 -\varepsilon)k$ then
  $G'=G-H$ contains a Hamiltonian cycle.
\end{thm}

One can also see that if $G$ is an exceptional strongly regular graph
then the number of Hamiltonian cycles in $G$ is very large .

%One can also see that if $G$ is an exceptional strongly regular graph
%then the number of Hamiltonian cycles in $G$ is $v {!} (\frac{k}{v})^v %\left(1+0(1)\right)^v$.

This follows from Theorem~\ref{th:03} and the following result of Krivelevich.~\cite{Kr}

\begin{thm}[Krivelevich]\label{th:02.4}
  Let $G$ be a $k$-regular graph on $v$ vertices such that
  \begin{itemize}
  \item[(i)] $\frac{k}{\lambda}\geq (\log v)^{1+\varepsilon}$ for some $\varepsilon>0$
  \item[(ii)] $\log k \cdot\log\frac{k}{\lambda}\gg\log v$
  \end{itemize}
  Then the number of Hamiltonian cycles in $G$ is $v {!} (\frac{k}{v})^v \left(1+o(1)\right)^v$.
\end{thm}

The \textit{toughness\/} of a finite connected graph
$G$ with vertex set $V$ is defined as the minimum of
the quotient $|X|/ c(V\backslash X)$ over all subsets
$X$ of $V$ such that $c(V\backslash X) > 1$ where
$c(Y)$ denotes the number of connected components of
the graph induced on~$Y$ by $G$.

Improving an earlier result of Alon, Brouwer~\cite{Br}
proved that if $G$ is a connected noncomplete
$k-$regular graph then its toughness $t$ satisfies $t >
k/\Lambda -2$. Hence Theorem~\ref{th:03} implies that
exceptional strongly regular graphs are very tough.

Chvatal conjectured long ago that $t$-tough graphs are
Hamiltonian if $t$ is a large enough constant.
Motivated by the above discussion we suggest the
following weaker problem: {\sl Prove that $\log v$ tough
graphs are Hamiltonian if the number of vertices $v$ is
sufficiently large!}

As we saw above the pseudorandomness of exceptional
strongly regular graphs implies that these graphs have
some interesting graph-theoretic properties, for many
more see~\cite{KS2}. It would be interesting to see
whether connected non-exceptional strongly regular
graphs also posses similar properties e.g. whether an appropriate analogue of the conclusion of Theorem~\ref{th:23} and Theorem~\ref{th:02.4} holds for them.

\section{More pseudo-random graphs}\label{sec:three}

In the last section we will show that in two other
natural classes of $k$-regular graphs $G\quad
k/\Lambda(G)$ is often relatively large.

A graph is called \textit{distance regular\/} if, given
any two vertices $g$ and $h$, the number of vertices in
$G$ at distance $i$ from $g$ and distance $j$ from $h$
only depends on $i$, $j$ and the distance between $g$
and $h$. As noted before a connected strongly regular graph
is a distance regular graph of diameter two.

For any graph $G$ with vertex set $V(G)$ let $G_r$
denote the $r$-th distance graph of $G$ i.e. the graph
on the same vertex set in which two vertices $u$ and
$v$ in $V(G)$ are connected if their distance in $G$ is
exactly $r$. We call $M$ a \textit{merged\/} graph of $G$
if it is the union of some of the $G_r$.

Denote the adjacency matrix of $G_r$ by $A_r$. It is
well-known see (e.g.~\cite{Go2}) that if $G$ is a
distance regular graph then $A_r$ is a polynomial of
degree $r$ in $A=A_1$, say $A_r=p_r(A)$. It follows
that if $\theta$ is an eigenvalue of $A$ with
eigenspace $V_\theta$, then each vector $v$ in
$V_\theta$ is an eigenvector of $A_r$ with eigenvalue
$p_r(\theta)$.  This implies that if $M$ is a merged
graph of $G$ and $\lambda$ is an eigenvalue of $M$ with
eigenspace $V_\lambda$, then $V_\lambda$ is the direct
sum of some eigenspaces of $G$.

We will use the following theorem of Godsi~\cite{Go1}.

\begin{thm}[Godsil]\label{th:31}
  Let $G$ be a connected distance regular graph of
  degree $k$\  $(k > 2)$. Suppose that $G$ is not a
  complete multipartite graph and that $\theta$ is an
  eigenvalue of $G$ with multiplicity $m$. Then if
  $\theta \neq \pm k$, the diameter of $G$ is at most
  $3m-4$ and $k \leq (m-1)(m-2)/2$.
\end{thm}

In particular we have $|G| < m^{6m}$

\smallskip

We prove the following.

\begin{cor}\label{co:32}
  Let $G$ be a connected non-bipartite distance-regular
  graph on $v$ vertices which is not a complete
  multipartite graph. Let $M$ be a connected merged
  graph of $G$ of degree $p$ . Then $\Lambda(M) <
  \sqrt{ 6pv \log\log v/\log v}$.
\end{cor}

\begin{proof}
  Let $k$ be the degree of $G$ and $m$ the smallest
  multiplicity of an eigenvalue $\theta \neq k$ of $G$.
  By the above discussion the multiplicity $t$ of an
  eigenvalue $\lambda \neq p$ of $M$ is at least as
  large as $m$. By Theorem~\ref{th:31} we have $t > m
  \geq \log v/6\log\log v$.  Similarly to the proof of
  Proposition~\ref{pr:11}~(i) we see that $t\lambda^2 <
  tr(\Adj(M)^2) =pv$, where $\Adj(M)$ is the adjacency
  matrix of $M$. Hence $\lambda^2 < 6pv \log\log v/\log
  v$ and our statement follows.
\end{proof}

In particular if we have a family of merged graphs
$M_n$ as above satisfying $p_n> \varepsilon v_n$ for
some $\varepsilon > 0$, then $\Lambda(M_n)=o(p_n)$.
This extends the main result of~\cite{Ni} to distance
regular graphs.

These results yield some partial evidence supporting Conjecture~\ref{conj:02} from the introduction. By a result of Bang, Dubickas, Koolen and Moulton~\cite{BDKM}
there are only finitely many distance regular graphs of fixed valency
greater than two. Hence in Conjecture~\ref{conj:02} we may assume that the valency is sufficiently large.

The following weaker form of Conjecture~\ref{conj:02} seems to be already very hard

\textbf{Conjecture 3.1}
Connected distance regular graphs of diameter d are Hamiltonian with finitely many exceptions.

 Note that van Dam and Koolen ~\cite{DK} constructed a family of non-vertex transitive distance regular graphs with unbounded diameter.
 However such families seem to be hard to find. It may be a reasonable first step to establish the validity of Conjecture 3.1 for distance transitive graphs.

\smallskip

Finally we will prove Theorem~\ref{th:05} in a somewhat
more general form.

Let $P$ be a permutation group of degree $n$. If $V$ is
an $n$ dimensional vector-space over the complex
numbers then $P$ acts on $V$ in a natural way by
permuting the elements of an orthonormal basis.  Hence
we have a complex representation $\pi$ of the group $P$
by permutation matrices. If the group $G$ is transitive
then it is well-known (see~\cite[Corollary 5.15]{Is})
that the trivial representation $1_P$ occurs with
multiplicity~$1$ in a decomposition of $\pi$ into
irreducible constituents.

The following result is motivated by an argument of
Gowers~\cite{Gow}.

\begin{prop}\label{pr:3.4}
  Let $G$ be a $k$-regular connected graph on $v$
  vertices and $P$ a group of automorphisms which acts
  transitively on $G$. Assume that the minimal degree
  of a non-trivial irreducible constituent of the
  natural representation $\pi$ of $P$ by permutation
  matrices is at least $t$. Then $\Lambda(G) <
  \sqrt{vk/t}$.
\end{prop}

\begin{proof}
  It is well-known that the permutation matrices
  $\pi(p)$ ($p\in P$) commute with the $\Adj(G)$. This
  implies that the eigenspaces of $\Adj(G)$ are
  invariant under $\pi(P)$ (see~\cite[1.5]{Ba}). Since
  $G$ is connected the eigenspace corresponding to $k$
  is the one-dimensional subspace $V_k=\{<x,x\dots
  x>\}$.  $V_k$ affords the trivial representation
  $1_P$. By our condition the dimensions of all other
  eigenspaces are larger than $t$. It follows that the
  multiplicity of an eigenvalue $\lambda$ different
  from $k$ is at least $t$. As above we have
  $t\lambda^2 < \Tr(\Adj(G)^2)= vk$. Our statement
  follows.
\end{proof}

Recall that a permutation group $P$ acting on a set
$\Omega$ is \textit{primitive\/} if there is no
non-trivial $P$ invariant partition of $\Omega$ or
equivalently if all the graphs with vertex set $\Omega$
invariant under $P$ are connected~\cite{Cam}.

\begin{cor}\label{co:3.5}
  Let $G$ be a $k$-regular graph on $v$ vertices such
  that its automorphism group $A$ is a primitive
  permutation group. If $S$ is the largest abelian
  normal subgroup of $A$ then $\Lambda(G) < c_1
  (vk/\sqrt{\log(A/S)} )^{1/2}$ for some constant $c_1
  > 0$.
\end{cor}

\begin{proof}
  It is well-known (see~\cite[Exercise 2.20]{Cam}) that
  if $P$ is a primitive group then a non-trivial
  irreducible constituent of the natural representation
  $\pi$ by permutation matrices is faithful.

  By a classical theorem of Jordan a finite subgroup of
  $GL(n,C)$ has an abelian normal subgroup of index
  $J(n)$, where $J(n) < c^{n^2}$ for some constant $c$
  (see~\cite[Theorem 14.12]{Is}).

  Hence if $t$ is the smallest degree of a non-trivial
  irreducible constituent of $\pi$ then we have $t >
  c_0 \sqrt{\log |A/S|}$. Using Proposition~\ref{pr:3.4} we
  obtain our statement.
\end{proof}

It would be interesting to see whether Corollary~\ref{co:3.5} has some kind of extension to merged graphs of primitive coherent configurations (see~\cite{Ba}).
Theorem~\ref{th:05} is an immediate consequence of
Corollary~\ref{co:3.5}.

We remark that if $A$ is a primitive permutation group
of degree $v$ then it has a unique largest abelian
normal subgroup $S$ which is either trivial or
elementary abelian of order $v$. In the first case we
have $\Lambda(G) < c_1 k \sqrt{v/k\log v}$. In the
second case if $A/S$ is small then a graph $G$
invariant under $A$ is "close" to being a Cayley graph
of an elementary abelian group.

\end{document}